\documentclass[12pt]{amsproc}
  \usepackage{latexsym} 
  \usepackage[all]{xy}
  \usepackage{amsfonts} 
  \usepackage{amsthm} 
  \usepackage{amsmath} 
  \usepackage{amssymb}
  \usepackage{pifont}

  \def\<{{\langle}} 
  \def\>{{\rangle}} 
   
  \def\la{{\triangleright}} 
  \def\eps{\varepsilon}

  \def\note#1{{}}

  \def\note#1{}

  \def\lhom#1#2#3{{}{\sb{#1}{\rm Hom}(#2,#3)}} 
  \def\rhom#1#2#3{{{\rm Hom}\sb{#1}(#2,#3)}}

  \def\rend#1#2{{{\rm End}\sb{#1}(#2)}}

  \def\R{\mathbb{R}}

  \def\beq{\begin{equation}} 
  \def\eeq{\end{equation}}

  \def\id{\mathrm{id}}

  \def\ot{{\otimes}}

  \def\tsi{\tilde{\sigma}}

  \def\hsi{\hat{\sigma}}
   \def\bsi{\bar{\sigma}}
    \def\tsi{\tilde{\sigma}}

  \def\coker{\mathrm{coker}}

    \def\oan#1{\Omega^{#1} (A)}

\def\1{\mathbb{I}}

\def\k{\Bbbk}

  \newcounter{zlist} 
  \newenvironment{zlist}{\begin{list}{(\arabic{zlist})}{ 
  \usecounter{zlist}\leftmargin2.5em\labelwidth2em\labelsep0.5em 
  \topsep0.6ex
  \parsep0.3ex plus0.2ex minus0.1ex}}{\end{list}}

  \newcounter{blist}

  \newcounter{rlist}

\def\stac#1{\raise-.2cm\hbox{$\stackrel{\displaystyle\otimes}{\scriptscriptstyle{#1}}$}}

\def\cten#1{\raise-.2cm\hbox{$\stackrel{\displaystyle\widehat{\otimes}}
{\scriptscriptstyle{#1}}$}}

  \def\Label#1{\label{#1}\ifmmode\llap{[#1] }\else 
  \marginpar{\smash{\hbox{\tiny [#1]}}}\fi} 
  \def\Label{\label}

  \newtheorem{proposition}{Proposition}[section]
  \newtheorem{lemma}[proposition]{Lemma} 
  \newtheorem{corollary}[proposition]{Corollary} 
  \newtheorem{theorem}[proposition]{Theorem} 

  \theoremstyle{definition} 
  \newtheorem{definition}[proposition]{Definition}

  \theoremstyle{remark}

\begin{document} 

 \title{Divergences on projective modules and non-commutative integrals} 
 \author{Tomasz Brzezi\'nski}
 \address{ Department of Mathematics, Swansea University, 
  Singleton Park, \newline\indent  Swansea SA2 8PP, U.K.} 
  \email{T.Brzezinski@swansea.ac.uk}   
     \date{\today} 

  \begin{abstract} A method of constructing (finitely generated and projective) right module structure on a finitely generated projective left module over an algebra is presented. This leads to a construction of a first order differential calculus on such a module which admits a hom-connection or a divergence. Properties of integrals associated to this divergence are studied, in particular the formula of integration by parts is derived. Specific examples include inner calculi on a noncommutative algebra, the Berezin integral on  the supercircle and  integrals on Hopf algebras. 
 \end{abstract} 
  \maketitle

\section{Introduction}
The notion of a {\em hom-connection} or a {\em divergence} over a non-commutative algebra $A$ was introduced in \cite{Brz:con}. In \cite{BrzElK:int} a construction of differential calculi which admit hom-connections was presented. This construction is based on the use of {\em twisted multi-derivations}, and the constructed first-order calculus $\oan 1$ is free as  a left and right $A$-module. Since $\oan 1$ should be understood as a module of sections on the cotangent bundle over a manifold represented by $A$, the construction presented in \cite{BrzElK:int} corresponds to parallelisable manifolds or to an algebra of functions on a local chart. In general, however, one should expect $\oan 1$  to be a finitely generated and projective module over $A$ (thus corresponding to sections of a non-trivial vector bundle by the Serre-Swan theorem). The first aim  of these notes is to extend the construction in \cite{BrzElK:int} to finitely generated and projective modules.

It has been argued in \cite{BrzElK:int} that the canonical map defining the cokernel of a hom-connection should be interpreted as an integral. The motivation for this intepretation comes from \cite{BerLei:int}, where the integral on supermanifolds or the Berezin integral is shown to originate from a map which can be intepreted as a hom-connection termed a {\em right connection} in \cite{Man:gau}. In the second part of these notes we show that the integral associated to hom-connections constructed in \cite{BrzElK:int} allows one to use the integration by parts method.  This is further indication that indeed the cokernel of a hom-connection should be understood as a non-commutative integral. The final aim of these notes is to present examples of integrals and applications of the integration by parts formula. These include integrals associated to an inner calculus, properties of integrals on Hopf algebras, and the integral on the algebra of functions on a supercircle, which is explicitly shown to coincide with the Berezin integral as expected. 

The basic notion studied in these notes is that of a {\em hom-connection} or a {\em divergence} which we recall presently. Let $A$ be an algebra over a field $\k$. A first order differential calculus over $A$ is an  $A$-bimodule $\oan 1$ together with a map $d: A\to \oan 1$ which satisfies the Leibniz rule (i.e.\ $d$ is an $\oan 1$-valued derivation on $A$). A {\em divergence} on $A$ is a $\k$-linear map $\nabla: \rhom A {\oan 1} A \to A$ such that, for all $f\in \rhom A {\oan 1} A$ and $a\in A$,
\begin{displaymath}
\nabla (f a) = \nabla(f) a + f(da), 
\end{displaymath}
where $fa$ is a right $A$-module map $\oan 1\to A$ given by $\omega\mapsto f(a\omega)$. Here and below, given an $A$-bimodule $M$, $\rhom A M A$ denotes the vector space of all right $A$-module maps $M\to A$ and $\lhom A MA$ that of left $A$-module maps. In terminology of \cite{Brz:con} a divergence $\nabla$ is the same as a {\em hom-connection} $(\nabla, A)$ with respect to differential graded algebra with degree-one part $\oan 1$.

\section{Divergences on projective modules}\label{sec.div.proj}
\setcounter{equation}{0}
By a {\em right twisted multi-derivation} in an algebra $A$ we mean a pair $(\partial, \sigma)$, where $\sigma: A\to M_n(A)$  ($M_n(A)$ is the algebra of $n\times n$ matrices with entries from $A$) and $\partial :A\to A^n$ are $\k $-linear maps such that, 
for all $a\in A$, $b\in B$,
\begin{equation}\label{eq.partial}
\partial(ab) = \partial(a)\sigma(b) + a\partial(b).
\end{equation}
Here $A^n$ is understood as an $(A$-$M_n(A))$-bimodule. 
Thus, writing $\sigma(a) = (\sigma_{ij}(a))_{i,j=1}^n$ and $\partial(a) = (\partial_i(a))_{i=1}^n$, 
 \eqref{eq.partial} is equivalent to the following $n$ equations
$$
\partial_i(ab) = \sum_j \partial_j(a)\sigma_{ji}(b) + a\partial_i(b), \qquad i=1,2,\ldots, n.
$$
In contrast to \cite{BrzElK:int} we do not require $\sigma$ be an algebra map.

The definition of a twisted multi-derivation $(\partial, \sigma)$ above is designed in such a way as to lead to a first order differential calculus. 
\begin{lemma}\label{lem.dif.calc}
Let $(\partial: A\to A^n, \sigma)$ be a right twisted multi-derivation, and let $M$ be an $A$-bimodule in which there are elements $\omega_1,\omega_2, \ldots, \omega_n$ such that 
\begin{equation}\label{eq.dif.rel0}
 \omega_i a = \sum _j \sigma_{ij}(a) \omega_j ,  \qquad i=1,2,\ldots, n.
\end{equation}
Then $\oan 1 = M$ is a first order differential calculus on $A$ with exterior differential
 \begin{equation}\label{eq.dif.d}
da = \sum_i \partial_i(a)\omega_i ,
\end{equation}
for all $a\in A$. 
\end{lemma}
\begin{proof}
That $d$ satisfies the Leibniz rule follows immediately from the twisted derivation property of $\partial$ and relations \eqref{eq.dif.rel0}. 
\end{proof}

In view of Lemma~\ref{lem.dif.calc} a right twisted multi-derivation $(\partial,\sigma)$ can be understood as a derivation in the bimodule $M$ in which relations \eqref{eq.dif.rel0}  hold. We now proceed to construct a class of such modules. 

A map $\sigma : A \to M_n(A)$ can be  viewed as an element of $M_n(\rend \k  A)$ . We write $\bullet$ for the product in $M_n(\rend \k  A)$ and  $\sigma^T$ for the transpose of $\sigma$ in $M_n(\rend \k  A)$. Furthermore, an element $\pi$ of $M_n(A)$  is understood as an element of  $M_n(\rend \k  A)$ by right multiplication, i.e.
$$
\pi_{ij} : a \mapsto a\pi_{ij}.
$$
In concordance with this  $\1$ denotes the identity matrix both in $M_n(A)$ and  in $M_n(\rend \k  A)$. 

\begin{definition}\label{def.proj.system}
For a $\k$-algebra $A$, let $\pi\in M_n(A)$ be an idempotent and let $\sigma, \tsi: A\to M_n(A)$ be  $\k$-linear maps such that, for all $a,b\in A$,
\begin{equation}\label{tsi}
\tsi(b) \sigma(a)\pi = \tsi(ba)\pi, \qquad
\tsi(1) = \pi, 
\end{equation}
and
\begin{equation}\label{pi}
\sigma(a)\pi = \pi \tsi(a),
\end{equation}
(all equations in $M_n(A)$).
The triple $(\pi,\sigma,\tsi)$ is called a {\em pre-projective system} on $A$. 

A pre-projective system $(\pi,\sigma,\tsi)$ is called a {\em projective system} 
provided there exists an algebra map $\bsi: A \to M_n(A)$  such that
\begin{equation}\label{bar.tsi}
\bsi \bullet \tsi^T = \1, \qquad \sigma^T\bullet\bsi = \pi \, ,
\end{equation}
in $M_n(\rend \k  A)$. 
A projective system is denoted by $(\pi,\sigma,\tsi;\bsi)$.
\end{definition}
\begin{lemma}\label{lem.proj.system}
Let $(\pi,\sigma,\tsi)$ be a pre-projective system on $A$. Then, for all $a\in A$,
\begin{equation}\label{pro1}
\pi \sigma(a) \pi = \tsi(a)\pi,
\qquad 
\pi \tsi(a)  = \tsi(a)\pi,
\end{equation}
\begin{equation}\label{pro3}
\pi \tsi(a) \pi = \tsi(a)\pi,
\qquad
\sigma(a) \pi = \tsi(a)\pi
\end{equation}
and
\begin{equation}\label{pro5}
\pi \sigma(a) \pi = \sigma(a)\pi .
\end{equation}
Furthermore, for all $a,b\in A$,
\begin{equation}\label{tsi.alg}
\tsi(b) \tsi(a)\pi = \tsi(ba)\pi =  \tsi(b) \pi \tsi(a),
\end{equation}
and 
\begin{equation}\label{sig.alg}
\sigma(b) \sigma(a)\pi = \sigma(ba)\pi , \qquad \sigma(1)\pi = \pi.
\end{equation}
\end{lemma}
\begin{proof}
The first of equalities \eqref{pro1} follows by evaluating of \eqref{tsi} at $b=1$. The second of equalities  \eqref{pro1} is a consequence of the first one, equation \eqref{pi} and the idempotent property of $\pi$.  Multiplying the second of equalities  \eqref{pro1} by $\pi$ and using the idempotent property of $\pi$ one obtains the fist of equalities \eqref{pro3}. Then the use of \eqref{pi} yields the second of \eqref{pro3}. Next, \eqref{pro5} is a consequence of the first of  equalities \eqref{pro1} and the second one of \eqref{pro3}. The equalities \eqref{tsi.alg} are a simple consequence of the second of \eqref{pro3} and the first of \eqref{tsi}, and of the second of \eqref{pro1}. Finally the second of \eqref{sig.alg} follows by the normalisation of $\tsi$ and the idempotent property of $\pi$, while the first of \eqref{sig.alg} can be checked by a straightforward calculation that uses for example \eqref{pi}, the second of \eqref{pro3} and \eqref{tsi.alg}.
\end{proof}

In view of the above properties and, in particular, in view of the fact that $\tsi$ coincides with $\sigma$ when multiplied by $\pi$, the map $\tsi$ plays an auxiliary role; without losing much generality, one might assume that $\tsi =\sigma$ in the pre-projective system from the onset.

\begin{proposition}\label{prop.bimod}
Let $M$ be a finitely generated projective left $A$-module with dual basis $\omega_i\in M$, $\eta_i\in \lhom A MA$. Write $\pi = (\pi_{ij} = \eta_j(\omega_i))\in M_n(A)$ for the idempotent associated to this dual basis. Assume that there exist $\k$-linear maps $\sigma,\tsi: A\to M_n(A)$, such that $(\pi,\sigma,\tsi)$ is a pre-projective system on $A$. Then:
\begin{zlist} 
\item $M$ is an $A$-bimodule with the right multiplication
\begin{equation}\label{right}
ma := \sum_{i,j} \eta_i(m) \tsi_{ij}(a) \omega_j = \sum_{i,j} \eta_i(m) \sigma_{ij}(a) \omega_j ,
\end{equation}
for all $a\in A$, $m\in M$. In particular, for all $a\in A$, 
\begin{equation}\label{eq.dif.rel}
 \omega_i a = \sum _j \sigma_{ij}(a) \omega_j= \sum _j \tsi_{ij}(a) \omega_j ,  \qquad i=1,2,\ldots, n.
\end{equation}
\item If there exists an algebra map $\bsi: A\to M_n(A)$ such that $(\pi,\sigma,\tsi; \bsi)$ is a projective system, then 
\begin{equation} \label{rel.dif.bar}
a\omega_i = \sum_{j} \omega_j\bsi_{ji}(a), \qquad \mbox{for all} \quad a\in A, i=1,2,\ldots n,
\end{equation}
and $M$ is a finitely generated and projective right $A$-module. 
\end{zlist}
\end{proposition}
\begin{proof}
(1) The unitality of action \eqref{right} follows by the normalisation of $\tsi$ (i.e.\ by the second of equations \eqref{tsi}) and by the dual basis property. The associativity is checked by the following calculation
\begin{eqnarray*}
(ma)b & = & \sum_{i,j,k,l} \eta_k\left(\eta_i(m) \tsi_{ij}(a) \omega_j\right)\tsi_{kl}(b)\omega_l \\
&=& \sum_{i,j,k,l} \eta_i(m) \tsi_{ij}(a) \pi_{jk}\tsi_{kl}(b)\omega_l 
= \sum_{i,k} \eta_i(m) \tsi_{ik}(ab)\omega_k = m(ab).
\end{eqnarray*}
The second equality is a consequence of left $A$-linearity of the $\eta_i$, and the third one follows by  \eqref{tsi.alg}.
Therefore, the right action \eqref{right} is associative. It is also left $A$-linear since all the $\eta_i$ are left $A$-linear. The relations \eqref{eq.dif.rel} are an immediate consequence of the definition of a right $A$-multiplication,  \eqref{pi} and the dual basis property. The second equalities in \eqref{right} and \eqref{eq.dif.rel} follow by the second of equalities \eqref{pro3} and the dual basis property which in particular affirms that $\omega_i = \sum_j \pi_{ij}\omega_j$.

(2) In view of the dual basis property, the second of equations \eqref{bar.tsi} and relations  \eqref{eq.dif.rel} one can compute, for all $a\in A$,
$$
a\omega_i= \sum_j a\pi_{ij}\omega_j = \sum_{j,k} \sigma_{kj}\left(\bsi_{ki}(a)\right)\omega_j =
 \sum_{j}  \omega_j\bsi_{ji}(a).
 $$
 This proves the identities \eqref{rel.dif.bar}.

Define $\k$-linear maps $\xi_i : M\to A$ by
\begin{equation}\label{xi}
\xi _i = \sum_j \bsi_{ij}\circ \eta_j, \qquad i=1,2,\ldots, n.
\end{equation}
Then, for all $m\in M$,
\[
\sum_i \omega_i\xi_i(m) = \sum_{i,k} \omega_i\bsi_{ik}(\eta_k(m))
= \sum_k \eta_k(m) \omega_k = m,
\]
where the second equality follows by \eqref{rel.dif.bar}.
 This proves that $M$ is generated by the $\omega_i$ as a right $A$-module. 

Next, for all $a\in A$, $m\in M$ and $i=1,2,\ldots, n$,
\begin{eqnarray*}
\xi_i(ma) &=& \sum_{j,k,l} \bsi_{il}\left(\eta_l\left(\eta_j(m)\tsi_{jk}(a)\omega_k\right)\right)  = \sum_{j,k,l} \bsi_{il}\left(\eta_j(m)\tsi_{jk}(a)\pi_{kl}\right)\\
&=& \sum_{j,k,l} \bsi_{il}\left(\eta_j(m)\pi_{jk}\tsi_{kl}(a)\right) = \sum_{k,l} \bsi_{il}\left(\eta_k(m)\tsi_{kl}(a)\right)\\
& = & \sum_{k,l,r} \bsi_{ir}\left(\eta_k(m)\right)\bsi_{rl}\left(\tsi_{kl}(a)\right) = \xi_i(m)a.
\end{eqnarray*}
The second equality follows by the left $A$-linearity of the $\eta_l$ (combined with the definition of the idempotent $\pi$). The third equality is a consequence of equations \eqref{pi}, while the fourth one is the dual basis property. To derive the last two equalities the fact that $\bsi$ is an algebra map and the first of equations \eqref{bar.tsi} were used. This proves that $\omega_i$, $\xi_i$ form a dual basis for the right $A$-module $M$.
\end{proof}

\begin{corollary}\label{cor.dif}
In the setup of Proposition~\ref{prop.bimod}(1) assume that there is a right twisted multiderivation $(\partial: A\to A^n,\sigma)$. Then $\oan 1 = M$ is a first order differential calculus on $A$ with exterior differential defined by the formula \eqref{eq.dif.d}. 
\end{corollary}
\begin{proof}
 Since the dual basis elements $\omega_i$ satisfy equations \eqref{eq.dif.rel0}, Lemma~\ref{lem.dif.calc} yields the assertion. 
\end{proof}

\begin{lemma}\label{lem.dif.pi}
Let $(\partial: A\to A^n, \sigma)$ be a right twisted multi-derivation and assume that there exists $\pi\in M_n(A)$ and $\tsi: A\to M_n(A)$ satisfying equations \eqref{pi}. Then 
\[
\partial^\pi : A\to A^n, \qquad a\mapsto \partial(a)\pi,
\]
is a $\tsi$-twisted multi-derivation. Furthermore, in the situation 
of Proposition~\ref{prop.bimod}(1), the calculus $(M,d^\pi)$ induced by $(\partial^\pi, \tsi)$ by the procedure described in Lemma~\ref{lem.dif.calc} coincides with the calculus $(M,d)$ induced by $(\partial,\sigma)$.
\end{lemma}
\begin{proof}
Since $\pi$ intertwines $\sigma$ with $\tsi$ (i.e.\ equations \eqref{pi} are satisfied), the $\sigma$-twisted multi-derivation property of $\partial$ implies that $(\partial^\pi, \tsi)$ is a right twisted multi-derivation.  To prove the second statement, first observe that, in view of \eqref{eq.dif.rel}, 
the method described in Lemma~\ref{lem.dif.calc} can be applied also to $(\partial^\pi, \tsi)$. Now, use that $\pi$ is an idempotent corresponding to the dual basis $\omega_i, \eta_i$ for $M$ to find that $\sum_i \partial_i(a)\omega_i =\sum_i \partial^\pi_i(a)\omega_i$, as required.
\end{proof}

\begin{definition}\label{def.der.free}
Let $(\partial, \sigma)$ be a right twisted multi-derivation, where $\sigma$ is a part of a projective system $(\pi,\sigma,\tsi;\bsi)$. We say that $(\partial, \sigma)$ is {\em projectively free}, provided there exist a $\k$-linear map $\hsi: A \to M_n(A)$ such that
\begin{equation}\label{hat.sigma}
\hsi \bullet \bsi^T = \1, \qquad \bsi^T\bullet\hsi = \1\, ,
\end{equation}
in $M_n(\rend \k  A)$. 
The notation $(\partial: A\to A^n, \sigma; \tsi, \bsi ,\hsi;\pi)$ is used to record such a derivation. 

A projectively free right twisted multi-derivation $(\partial, \sigma; \tsi, \bsi ,\hsi;\pi)$ is said to be {\em free}   if $\pi = \1$ (and then, necessarily, $\tsi = \sigma$). This is recorded as $(\partial, \sigma; \bsi ,\hsi)$
\end{definition}
An easy exercise (left to the reader) allows one to establish the following

\begin{lemma}\label{lem.free} If $(\partial: A\to A^n, \sigma; \tsi, \bsi ,\hsi;\pi)$ is a projectively free right twisted multi-derivation, then $\hsi$ is an algebra map.
\end{lemma}

In the case of a projectively free multiderivation, more can be said about the bimodule structure of a projective module induced by $\pi$.

\begin{lemma}\label{lem.dual}
In the setup of Proposition~\ref{prop.bimod}(2) assume that there exists an algebra map $\hsi: A\to M_n(A)$ that satisfies the second of equations \eqref{hat.sigma}. Then, for all $f\in \rhom AMA$,
\begin{equation}\label{dual}
f = \sum_{i,k} \xi_i \hsi_{ik}\left(f\left(\omega_k\right)\right),
\end{equation} 
where the $\xi_i$ are defined by equations \eqref{xi} in the proof of Proposition~\ref{prop.bimod}(2).
\end{lemma}
\begin{proof}
For all $m\in M$,
\begin{eqnarray*}
\sum_{i,k} \xi_i \hsi_{ik}\left(f\left(\omega_k\right)\right)(m) &=& \sum_{i,k,l} \bsi_{il}\left(\eta_l\left(\hsi_{ik}\left(f\left(\omega_k\right)\right)m\right)\right) = \sum_{i,k,l} \bsi_{il}\left(\hsi_{ik}\left(f\left(\omega_k\right)\right)\eta_l\left(m\right)\right)\\
&=& \sum_{i,k,l,r} \bsi_{ir}\left(\hsi_{ik}\left(f\left(\omega_k\right)\right)\right)\bsi_{rl}\left(\eta_l\left(m\right)\right) = \sum_{k,l} f\left(\omega_k\right)\bsi_{kl}\left(\eta_l\left(m\right)\right)\\
&=& \sum_{k,l} f\left(\omega_k \bsi_{kl}\left(\eta_l\left(m\right)\right)\right) = \sum_{k} f\left(\omega_k \xi_{k}\left( m\right)\right) = f(m).
\end{eqnarray*}
The first equality records the definitions of the $\xi_i$ and the right $A$-action on the dual module $\rhom AMA$, while the second equality uses the left $A$-linearity of the $\eta_l$. The multiplicativity of $\hsi$ is used in the third step followed by the second of equations \eqref{hat.sigma} in the fourth one. The fifth equality follows by the right $A$-linearity of $f$. Finally, the definition of the $\xi_k$ and the fact that $\omega_k$, $\xi_k$ form a dual basis for the right $A$-module $M$ are used in the derivation of the last two equalities.
\end{proof}

\begin{theorem}\label{thm.hom-der}
Let $(\partial: A\to A^n, \sigma; \tsi, \bsi ,\hsi;\pi)$  be a projectively free right twisted multi-derivation on $A$, and let $\oan 1=M$ be the associated first order differential calculus built on a finitely generated and projective left $A$-module $M$ with dual basis $\omega_i$, $\eta_i$ as described in Corollary~\ref{cor.dif}.  For each $i=1,2,\ldots, n$, write $\partial_i^\sigma := \sum_{j,\,k} \bsi_{kj}\circ\partial^\pi_j\circ \hsi_{ki}$, where the $\partial_j^\pi$ are described in Lemma~\ref{lem.dif.pi}, and define
\begin{equation}\label{def.hom.twist}
\nabla : \rhom A {\oan 1} A\to A, \qquad f\mapsto \sum_i \partial_i^\sigma \left( f\left(\omega_i\right)\right)\, .
\end{equation}
Then $\nabla$ is a divergence on $A$. 
\end{theorem}
\begin{proof} 
First we need the following
\begin{lemma}\label{lem.sigma}
$(\partial^\sigma,\hsi)$ is a right twisted multi-derivation on $A$.
\end{lemma}
\begin{proof}
For all $a,b\in A$,
\begin{eqnarray*}
\partial^\sigma_i (ab) &=& \sum_{j,k} \bsi_{kj}\left(\partial_j^\pi\left(\hsi_{ki}(ab)\right)\right) =
\sum_{j,k,l} \bsi_{kj}\left(\partial_j^\pi\left(\hsi_{kl}(a)\hsi_{li}(b)\right)\right) \\
&\!\!\!\!\!\!=&\!\!\!\!\!\! \sum_{j,k,l, r} \bsi_{kj}\left(\partial_r^\pi\left(\hsi_{kl}(a)\right)\tsi_{rj}\left(\hsi_{li}(b)\right)\right)  + \sum_{j,k,l} \bsi_{kj}\left(\hsi_{kl}(a)\partial_j^\pi\left(\hsi_{li}(b)\right)\right)
\\
&\!\!\!\!\!\!=&\!\!\!\!\!\! \sum_{j,k,l, r,s} \bsi_{ks}\left(\partial_r^\pi\left(\hsi_{kl}(a)\right)\right)\bsi_{sj}\left(\tsi_{rj}\left(\hsi_{li}(b)\right)\right)  + \sum_{j,k,l,s} \bsi_{ks}\left(\hsi_{kl}(a)\right)\bsi_{sj}\left(\partial_j^\pi\left(\hsi_{li}(b)\right)\right)\\
&\!\!\!\!\!\!=&\!\!\!\!\!\! \sum_{k,l, r} \bsi_{kr}\left(\partial_r^\pi\left(\hsi_{kl}(a)\right)\right)\hsi_{li}(b)  + \sum_{j,l} a\bsi_{lj}\left(\partial_j^\pi\left(\hsi_{li}(b)\right)\right) = \sum_{l} \partial^\sigma_l(a)\hsi_{li}(b)  + a\partial^\sigma_i(b).
\end{eqnarray*}
The second and fourth equalities follow by the multiplicativity of $\hsi$ and $\bsi$ respectively. The third equality is a consequence of the twisted derivation property of $\partial^\pi$; see Lemma~\ref{lem.dif.pi}. The first of equations \eqref{bar.tsi} and the second of \eqref{hat.sigma} lead to the fifth equality. The remainder is the definition of $\partial^\sigma$.
\end{proof}

With this lemma at hand, one can verify the Leibniz rule for $\nabla$ as follows. Take any $f\in \rhom A {\oan 1} A$ and $a\in A$ and compute
\begin{eqnarray*}
\nabla(fa) &=& \sum_i \partial_i^\sigma \left( fa\left(\omega_i\right)\right) =\sum_i \partial_i^\sigma \left( f\left(a\omega_i\right)\right)
= \sum_{i,j} \partial_i^\sigma \left( f(\omega_j)\bsi_{ji}(a)\right)\\
&=& \sum_{i,j,k} \partial_k^\sigma \left( f(\omega_j)\right)\hsi_{ki}\left(\bsi_{ji}(a)\right) +
 \sum_{i,j} f(\omega_j)\partial_i^\sigma \left( \bsi_{ji}(a)\right)\\
 &=&  \sum_i \partial_i^\sigma \left( f\left(\omega_i\right)\right) a + \sum_{i,j} f\left(\omega_j\bsi_{ji}\left(\partial^\pi_i\left(a\right)\right)\right)\\
 &=& \nabla(f)a + \sum_i f\left(\partial_i^\pi(a)\omega_i\right) 
 = \nabla(f)a + f(da),
\end{eqnarray*}
where the second equality is a consequence of the definition of the right action of $A$ on $f$. The third equality follows by \eqref{rel.dif.bar} and the  fact that $f$ is a right $A$-module map. 
The fourth equality is a consequence of Lemma~\ref{lem.sigma} and the next one uses the first of equations \eqref{hat.sigma}. The sixth equality follows by \eqref{rel.dif.bar}. The final equality follows by Lemma~\ref{lem.dif.pi}.
%
\end{proof}

If $A$ is an algebra of functions on the Euclidean space $\R^n$, $\partial_i$ are the standard partial derivatives and $\oan 1$ is the standard module of one-forms, then $\rhom A {\oan 1} A$ are simply vector fields and the formula \eqref{def.hom.twist} gives the classical divergence of the elementary vector calculus. This provides one with further justification of the name chosen for the abstract operation $\nabla$.

\begin{definition}\label{def.integral}
Let $\nabla$ be a divergence constructed in equation \eqref{def.hom.twist} in Theorem~\ref{thm.hom-der}. The cokernel map $\Lambda : A\to \coker \nabla$ is called the {\em integral on $A$ relative to $(\partial,  \sigma; \tsi, \bsi ,\hsi; \pi)$}, where $\pi$ is the idempotent matrix $\pi_{ij} = \eta_j(\omega_i)$.
\end{definition}

\begin{theorem}[Integration by parts]\label{thm.int.part}
Let $\Lambda$ be an integral on $A$ relative to a free multi-derivation $(\partial, \sigma; \bsi ,\hsi)$. Then, for all $i=1,2,\ldots, n$ and $a,b\in A$,
\begin{equation}\label{int.part}
\Lambda \left(a\partial^\sigma_i(b)\right) = - \sum_{l} \Lambda \left(\partial^\sigma_l(a)\hsi_{li}(b)\right).
\end{equation}
\end{theorem}
\begin{proof} 
Since it is assumed that $\pi_{ij} = \delta_{ij}$ the module $M = \oan 1$ is free (both as a left and right $A$-module) and the right $A$-module maps $\xi_i$ given by equations \eqref{xi} satisfy $\xi_i(\omega_j) = \delta_{ij}$.  For all $i=1,2,\ldots, n$ and $a\in A$,
\[
\nabla (a\xi_i) = \sum_j \partial_j^\sigma\left( a\xi_i\left(\omega_j\right)\right) = \partial_i^\sigma (a),
\]
which implies that $\Lambda (\partial_i^\sigma (a)) =0$, for all $a\in A$. Now apply Lemma~\ref{lem.sigma}  to the equality $\Lambda (\partial_i^\sigma (ab)) =0$.
\end{proof}

\section{Examples}
\setcounter{equation}{0}
\subsection{Inner calculi} In this section we construct an example of inner differential calculus based on twisted multi-derivations.
\begin{lemma}\label{lem.in1}
Let $\sigma: A \to M_n(A)$ be an algebra map. Fix $\delta = (\delta_1,\delta_2,\ldots, \delta_n) \in A^n$ and define
\begin{equation}\label{part.in}
\partial_i : A\to A, \qquad a\mapsto \sum_j \delta_j\sigma_{ji}(a) -a \delta_i, \qquad i = 1,2,\ldots, n,
\end{equation}
(i.e.\ $\partial : A\to A^n$, $\partial(a) = \delta\sigma(a) - a\delta$). Then $(\partial,\sigma)$ is a right twisted multi-derivation. 
\end{lemma}
\begin{proof}
For all $a,b\in A$,
\[
\partial(a)\sigma(b) + a\partial(b) = \delta\sigma(a)\sigma(b) - a\delta\sigma(b) + a\delta\sigma(b) - ab\delta = \delta\sigma(ab) - ab\delta = \partial(ab),
\]
since $\sigma$ is an algebra map.
\end{proof}
\begin{lemma}\label{lem.in2}
Let $M$  be as in Proposition~\ref{prop.bimod}(1). The calculus with $\oan 1=M$ and $d$ as in Corollary~\ref{cor.dif} associated to the right twisted multi-derivation $(\partial,\sigma)$ of Lemma~\ref{lem.in1} is inner.
\end{lemma}
\begin{proof}
Let $D= \sum_i\delta_i\omega_i$, where $\omega_i$ are the generators of the left $A$-module $M$. In view of the relations \eqref{eq.dif.rel}, one can compute
$$
da = \sum_{i,j}  \delta_j\sigma_{ji}(a)\omega_i -\sum_i a \delta_i\omega_i = \sum_i\delta_i\omega_i a - a\sum_i\delta_i\omega_i = [D,a],
$$
hence the calculus is inner as required.
\end{proof} 
\begin{lemma}
Let $\Lambda$ be the integral on $A$ relative to the free multi-derivation $(\partial, \sigma; \bsi ,\hsi)$, where $(\partial,\sigma)$ is as in Lemma~\ref{lem.in1}. Then, for all $a\in A$, $i=1,2,\ldots, n$,
$$
\Lambda\left(\sum_{k,l}\bsi_{kl}(\delta_l)\hsi_{ki}(a)\right) = \Lambda\left(\sum_l a\bsi_{il}(\delta_l)\right).
$$
\end{lemma}
\begin{proof}
For all $a\in A$, compute
\begin{eqnarray*}
\partial^\sigma_i(a) &=& \sum_{j,k,l} \bsi_{kj}\left(\delta_l \sigma_{lj}\left(\hsi_{ki}(a)\right)\right) - \sum_{j,k} \bsi_{kj}\left(\hsi_{ki}(a)\delta_j \right)\\
&=& \sum_{j,k,l,r} \bsi_{kr}\left(\delta_l \right)\bsi_{rj}\left(\sigma_{lj}\left(\hsi_{ki}(a)\right)\right) - \sum_{j,k,r} \bsi_{kr}\left(\hsi_{ki}(a)\right) \bsi_{rj}\left(\delta_j \right)\\
&=& \sum_{k,l}\bsi_{kl}(\delta_l)\hsi_{ki}(a) - \sum_l a\bsi_{il}(\delta_l).
\end{eqnarray*}
The second equality follows by the multiplicativity of $\bsi$, while the third one is a consequence of equations \eqref{bar.tsi} and \eqref{hat.sigma}. By the same arguments as in the proof of Theorem~\ref{thm.int.part}, $\Lambda(\partial^\sigma_i(a)) =0$, and the assertion follows.
\end{proof}

\subsection{Right integrals on Hopf algebras}
As explained in \cite[Section~4]{BrzElK:int} a free right twisted multi-derivation $(\partial, \sigma; \bsi ,\hsi)$ leading to an integral $\Lambda$ can be associated to any left covariant differential calculus on a quantum group or Hopf algebra $A$ with a bijective antipode $S$. Let $\Delta$ be the comultiplication  and $\eps$ the counit on $A$. Following \cite{Wor:dif}, a left-covariant differential calculus on $A$ is determined by $\k $-linear maps $\theta_{ij}, \chi_i: A\to \k $, $i,j=1,2,\ldots ,n$, which satisfy the following relations, for all $a,b\in A$,
\begin{equation} \label{eq.theta}
\theta_{ij}(ab) = \sum _k \theta_{ik}(a)\theta_{kj}(b), \qquad \theta_{ij}(1) = \delta_{ij},
\end{equation}
\begin{equation}\label{eq.chi}
\chi_i(ab) = \sum _j \chi_j(a) \theta_{ji}(b) + \eps(a)\chi_i(b).
\end{equation}
The dual space $A^* = \rhom \k  A \k $ is an algebra with convolution product  $f*g = (f\ot g)\circ \Delta$, which acts on $A$ from the left by $f\la a = (\id\ot f)(\Delta(a))$.  The datum \eqref{eq.theta}--\eqref{eq.chi} gives rise to the following free right twisted multi-derivation on $A$
\begin{equation}\label{qg.der}
\partial_i (a) = \chi_i \la a, \quad \sigma_{ij} (a) = \theta_{ij}\la a, \quad \bsi_{ij}(a) = (\theta_{ji}\circ S^{-1})\la a, \quad  \hsi_{ij}(a) = (\theta_{ij}\circ S^{-2})\la a.
\end{equation}
The module of one forms $\oan 1$ of a left-covariant differential calculus is free both as a left and right $A$-module (this reflects the fact that any Lie group is a parallelisable manifold), hence the idempotent matrix $\pi$ is trivial and Theorem~\ref{thm.int.part} yields
\begin{lemma}\label{lem.inv.calc}
Let $\theta_{ij}$, $\chi_i$, $i=1,2,\ldots ,n$ be linear maps satisfying equations \eqref{eq.theta}--\eqref{eq.chi}, and let $(\partial, \sigma; \bsi ,\hsi)$ be a free right-twisted multiderivation on $A$ given by equations \eqref{qg.der}. Denote by $\Lambda$  the  integral on $A$ relative to $(\partial, \sigma; \bsi ,\hsi)$. Then
\begin{equation}\label{lin.lambda}
\Lambda  \left(\left(\chi_i\circ S^{-2}\right) \la a\right) = \Lambda \left(S^2\left( \partial_i \left(a\right)\right)\right)= 0 ,
\end{equation}
for all $a\in A$ and $i=1,2,\ldots, n$. Consequently, if $\lambda: A\to \k$ is a right integral on a Hopf algebra $A$, then
\begin{equation}\label{rint}
\lambda \left(\left(\chi_i\circ S^{-2}\right) \la a\right) = \lambda \left(S^2\left( \partial_i \left(a\right)\right)\right)= 0 .
\end{equation}
\end{lemma}
\begin{proof}
By the same arguments as in the proof of \cite[Theorem~4.1]{BrzElK:int} one easily finds that $\partial_i^\sigma(a) = \left(\chi_i\circ S^{-2}\right) \la a$, and hence the first assertion follows by (the proof of) Theorem~\ref{thm.int.part}. Equations \eqref{rint} follow from \eqref{lin.lambda} and \cite[Theorem~4.1]{BrzElK:int}, whose second part asserts that any right integral on a Hopf algebra $A$ factors through $\Lambda$.
\end{proof}

\subsection{Integration on supermanifolds}
In this section we describe a baby example of a supermanifold, which makes explicit the connection between integrals defined in Defintion~\ref{def.integral} and the Berezin integral. Let $A$ be a superalgebra of (integrable) functions on the supercircle $S^{1\mid 1}$. That is, an element $a\in A$ is 
\begin{equation}\label{supercirc}
a(x,\theta) = a^0(x) + a^1(x)\theta,
\end{equation}
where $a^i: [0,1]\to \R$ are (integrable) functions such that $a^i(0) = a^i(1)$ and $\theta$ is a Grassmann variable, $\theta^2 =0$. The differentiation on $A$ is defined by
\begin{equation}\label{dif.circ}
\partial_xa(x,\theta) := \frac{da^0(x)}{dx} + \frac{da^1(x)}{dx} \theta, \qquad \partial_\theta a(x,\theta)  := a^1(x).
\end{equation}
One easily checks that $\partial = (\partial_x, \partial_\theta)$ is a right twisted multi-derivation with the diagonal $\sigma = \begin{pmatrix} \sigma_{xx} & 0 \cr 0 & \sigma_{\theta\theta} \end{pmatrix}$, where
\begin{equation}\label{sigma.circ}
\sigma_{xx}(a(x,\theta)) = a(x,\theta), \qquad \sigma_{\theta\theta}(a(x,\theta)) = a(x,-\theta).
\end{equation}
The twisted multi-derivation $(\partial,\sigma)$ is obviously free with $\bsi = \hsi = \sigma$.
\begin{lemma}\label{lem.super}
Let $\Lambda$ be the integral on $A$ relative to twisted multi-derivation \eqref{dif.circ}, \eqref{sigma.circ}. Then, for all $a(x,\theta) = a^0(x) + a^1(x)\theta$,
\begin{equation}\label{berezin}
\Lambda (a) = \int_0^1 a^1(x)dx,
\end{equation}
i.e.\ $\Lambda$ is the Berezin integral on the supercircle $S^{1\mid 1}$.
\end{lemma}
\begin{proof}
This can be argued as follows. $\oan 1$ is a free module generated by $dx$ and $d\theta$ hence any $f \in \rhom A{\oan 1} A$ is fully determined by its values on $dx$ and $d\theta$, i.e.\
\[
f_x(x,\theta) := f(dx)(x,\theta) = f_x^0(x) + f_x^1(x)\theta, \qquad f_\theta(x,\theta) := f(d\theta)(x,\theta) = f_\theta^0(x) + f_\theta^1(x)\theta.
\]
The divergence $\nabla$ comes out as
\begin{equation}\label{div.super}
\nabla(f)(x,\theta) = \partial_x f_x(x,\theta) - \partial_\theta f_\theta(x,\theta) = \frac{df_x^0(x)}{dx} - f_\theta^1(x) + \frac{df_\theta^0(x)}{dx}\theta.
\end{equation}
If $a(x,\theta)$ is purely even, i.e.\ $a(x,\theta) = a(x)$, then define $f_x(x,\theta) = 0$,
$f_\theta(x,\theta) = -a(x)\theta$. Equation \eqref{div.super} yields,
\[
a(x) = \nabla(f).
\]
Thus, the integral $\Lambda$ vanishes on the even part of $A$, i.e.\
\begin{equation}\label{lamb.zero}
\Lambda(a(x,\theta)) = \Lambda(a^0(x) + a^1(x)\theta) = \Lambda( a^1(x)\theta) .
\end{equation}
Since $\Lambda$ is the cokernel of $\nabla$, for all $f\in \rhom A{\oan 1} A$,
\[
0 = \Lambda\circ \nabla(f) = \Lambda\left(\frac{df_x^0(x)}{dx} - f_\theta^1(x) + \frac{df_\theta^0(x)}{dx}\theta\right) = \Lambda\left(\frac{d}{dx}f_\theta^0(x)\theta\right).
\]
On the other hand,  $f_\theta^0(0) = f_\theta^0(1)$, so $\int_0^1 \frac{d}{dx}f_\theta^0(x)dx =0$. By the universality of $\Lambda$, $ \Lambda(a^1(x)\theta) = \int_0^1 a^1(x)dx$, and in view of \eqref{lamb.zero},
\[
 \Lambda(a^0(x) + a^1(x)\theta) = \int_0^1 a^1(x)dx,
\]
as required.
\end{proof}

This example can easily be extended to all supermanifolds (with a compact, oriented body) thus explaining how integrals relative to the natural differential structure on a supermanifold yield Berezin integrals. 

\section*{Acknowledgements} 
I am grateful to Edwin Beggs for asking a question that started this work and to Ulrich Kr\"ahmer for inspiring comments.

\end{document}